\documentclass[11pt]{article}
\usepackage{amsmath,amssymb,amsthm}
\usepackage[margin=1in]{geometry}
\usepackage{enumitem}
\usepackage{tikz}

\theoremstyle{definition}
\newtheorem{definition}{Definition}
\newtheorem{problem}{Problem}
\theoremstyle{plain}
\newtheorem{lemma}{Lemma}
\newtheorem{proposition}{Proposition}
\newtheorem{theorem}{Theorem}
\newtheorem{corollary}{Corollary}
\theoremstyle{remark}
\newtheorem{remark}{Remark}

\newcommand{\Z}{\mathbb{Z}}
\newcommand{\popc}{\operatorname{pop}}
\newcommand{\col}{\operatorname{col}}

\definecolor{tgtok}{HTML}{2A2F37}
\definecolor{tgrows}{HTML}{2F6BD8}
\definecolor{tgcols}{HTML}{C0761A}
\definecolor{tggrid}{HTML}{B9C0CB}
\newcommand{\tglabels}{%
  \foreach \c in {0,1,2,3}{\node[font=\tiny,text=gray] at (\c+0.5,-4.34){c\c};}%
  \foreach \r in {0,1,2,3}{\node[font=\tiny,text=gray] at (-0.34,-\r-0.5){r\r};}}
\newcommand{\tgboard}[1]{%
  \begin{tikzpicture}[scale=0.62,baseline=(current bounding box.center)]
    \draw[tggrid] (0,0) grid (4,-4);\tglabels
    \foreach \r/\c in {#1}{\fill[tgtok] (\c+0.5,-\r-0.5) circle (0.30);}
  \end{tikzpicture}}

\usepackage[hidelinks]{hyperref}

\title{Searching for Primes: A Neural AlphaZero Approach to a Factoring Game}
\author{Marcel Crasmaru}
\date{\today}

\begin{document}
\maketitle

\begin{abstract}
We study a one-player token game on an $N\times N$ board where tokens slide along diagonals
or duplicate onto neighbouring ones to form a combinatorial rectangle $R\times S$. A conserved
integer weight $W'$ and a strict monovariant guarantee $O(N^2)$-length solutions, placing
the game in $\mathsf{NP}$. We prove that reaching a final position factors this $2N$-bit $W'$
into two $N$-bit factors $V,M < 2^{N}$ that encode the rectangle's rows and columns. Consequently, solving the game for a balanced-semiprime target is equivalent to
integer factoring. However, if the target rectangle is known, the solution reduces to two
polynomial-time steps: a forced downward chip-flow and a $0/1$-polynomial factorisation
leveraging Cohn's theorem. The game's entire difficulty is thus isolated to the initial
number-theoretic split. Supplying the popcounts of the factors as a promise preserves this
asymptotic hardness but bounds the target search space. We exploit this constrained space
using a learned policy/value network and an AlphaZero-style Monte-Carlo tree search,
empirically probing the limits of neural look-ahead on a factoring-equivalent environment.
\end{abstract}

\section{The game and the promise}

Write $[N]=\{0,1,\dots,N-1\}$ and index the board by cells $(r,c)\in[N]^2$, with
$r$ the row (top to bottom) and $c$ the column (left to right).  A cell holds at
most one token; a \emph{position} is the set $P\subseteq[N]^2$ of occupied cells.

The \emph{upper-left--to--lower-right diagonals} are the level sets of
$d(r,c)=r-c$; we call $d\in\{-(N-1),\dots,N-1\}$ the diagonal of the cell.  The
diagonal $d$ contains
\[
L_d \;=\; N-|d|
\]
cells.  Let $n_d(P)=|\{(r,c)\in P: r-c=d\}|$ be the number of tokens on diagonal
$d$, and let $\col_c(P)=\{r:(r,c)\in P\}$ be the pattern of column $c$.

\begin{definition}[Moves]\label{def:moves}
Two moves are allowed.
\begin{description}[leftmargin=2em]
\item[Slide.] Replace an occupied $(r,c)$ by any empty cell $(r',c')$ on the same
diagonal, $r'-c'=r-c$.
\item[Duplicate.] Remove an occupied cell on diagonal $d$ and place two tokens on
two distinct empty cells of diagonal $d-1$ (so a duplicate at $d$ requires
$n_{d-1}\le L_{d-1}-2$).
\end{description}
\end{definition}

We fix the orientation so that a duplicate lowers $d$ by one; the mirror
convention ($d\mapsto d+1$, weight $\sum 2^{\,c-r}$) is symmetric.

\begin{definition}[Final position; rectangle]
$P$ is \emph{final} if all non-empty columns have the same pattern:
$\col_c(P)=\col_{c'}(P)$ whenever both are non-empty.  A position is a
\emph{(combinatorial) rectangle} if $P=R\times S:=\{(r,c):r\in R,\ c\in S\}$ for
some $R,S\subseteq[N]$.
\end{definition}

Every rectangle is final.  Conversely a non-empty final position has one common
pattern $R$ shared by its non-empty columns $S$, so $P=R\times S$; thus
\emph{final $=$ rectangle}.

Figure~\ref{fig:tg25} illustrates the two moves and how a play ends, on a $3\times3$ board.

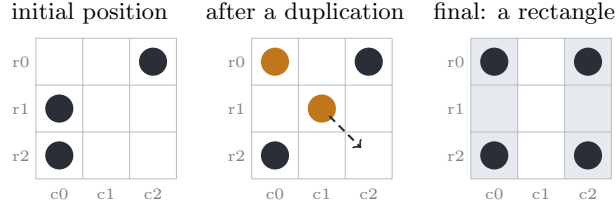
\begin{figure}[ht]
\centering
\begin{tabular}{c@{\quad}c@{\quad}c}
{\footnotesize initial position} & {\footnotesize after a duplication} & {\footnotesize final: a rectangle}\\[3pt]
\begin{tikzpicture}[scale=0.62,baseline=(current bounding box.center)]
  \draw[tggrid] (0,0) grid (3,-3);
  \foreach \c in {0,1,2}{\node[font=\tiny,text=gray] at (\c+0.5,-3.32){c\c};}
  \foreach \r in {0,1,2}{\node[font=\tiny,text=gray] at (-0.32,-\r-0.5){r\r};}
  \foreach \r/\c in {2/0,1/0,0/2}{\fill[tgtok] (\c+0.5,-\r-0.5) circle (0.3);}
\end{tikzpicture}
&
\begin{tikzpicture}[scale=0.62,baseline=(current bounding box.center)]
  \draw[tggrid] (0,0) grid (3,-3);
  \foreach \c in {0,1,2}{\node[font=\tiny,text=gray] at (\c+0.5,-3.32){c\c};}
  \foreach \r in {0,1,2}{\node[font=\tiny,text=gray] at (-0.32,-\r-0.5){r\r};}
  \foreach \r/\c in {0/2,2/0}{\fill[tgtok] (\c+0.5,-\r-0.5) circle (0.3);}
  \foreach \r/\c in {0/0,1/1}{\fill[tgcols] (\c+0.5,-\r-0.5) circle (0.3);}
  \draw[->,thick,densely dashed,tgtok] (1.66,-1.66) -- (2.34,-2.34);
\end{tikzpicture}
&
\begin{tikzpicture}[scale=0.62,baseline=(current bounding box.center)]
  \foreach \c in {0,2}{\fill[tggrid!35] (\c,0) rectangle (\c+1,-3);}
  \draw[tggrid] (0,0) grid (3,-3);
  \foreach \c in {0,1,2}{\node[font=\tiny,text=gray] at (\c+0.5,-3.32){c\c};}
  \foreach \r in {0,1,2}{\node[font=\tiny,text=gray] at (-0.32,-\r-0.5){r\r};}
  \foreach \r/\c in {0/0,0/2,2/0,2/2}{\fill[tgtok] (\c+0.5,-\r-0.5) circle (0.3);}
\end{tikzpicture}
\end{tabular}
\caption{How the game is played, on a $3\times3$ board. A token may \emph{slide} along its
upper-left--to--lower-right diagonal, or \emph{duplicate} onto the neighbouring diagonal,
splitting into two. From the position on the left, one duplication sends the middle token
onto the main diagonal as two tokens, and one slide (dashed) carries the off-place one to the
corner. The play then \emph{ends}: the two non-empty columns (shaded) carry the same
pattern---a rectangle---which is the winning condition. What such a final position
\emph{means} arithmetically is taken up in Section~\ref{sec:fact}.}
\label{fig:tg25}
\end{figure}

\begin{problem}[Rectangularise, with promise]\label{prob}
Given a position $P_0$ and integers $p,q$, under the promise that a final position
with exactly $p\cdot q$ tokens is reachable from $P_0$, output a sequence of moves
transforming $P_0$ into such a final position.
\end{problem}

Here $p=|R|$ and $q=|S|$ for the target rectangle $R\times S$.

\begin{remark}[Scope and motivation]
These choices simplify without softening the core. On a general (rectangular) board a
rectangle target encodes an \emph{arbitrary} two-factor split of the board weight, so the
game is as hard as ordinary integer factoring; the \emph{square} board and the
\emph{balanced semiprime} target (two $N$-bit primes) single out factoring's canonical hard
case, as in RSA. The popcounts $p=|R|$ and $q=|S|$ are then supplied as a \emph{promise}.
This does not make the instance easier---by Theorem~\ref{thm:hard} it stays as hard as
splitting a balanced semiprime---but it fixes the dimensions of the target rectangle and
thereby \emph{shrinks the search space} a learning agent must explore. That is exactly our
purpose: Sections~\ref{sec:learn}--\ref{sec:mcts} attack the promised problem with an
AlphaZero-style~\cite{alphazero} Monte-Carlo tree search guided by a learned policy/value
network, and the
popcount promise is the lever that turns target selection into a constrained search rather
than an open one.
\end{remark}

\section{Invariants}\label{sec:inv}

Give the token at $(r,c)$ the weight $2^{\,r-c}$ and set
$W(P)=\sum_{(r,c)\in P}2^{\,r-c}=\sum_d n_d(P)\,2^{d}$.

\begin{lemma}[Conservation]\label{lem:W}
Both moves preserve $W$.
\end{lemma}
\begin{proof}
A slide fixes $r-c$.  A duplicate removes $2^{d}$ and adds $2\cdot 2^{\,d-1}=2^{d}$.
\end{proof}

Since a slide can carry a token to \emph{any} empty cell of its diagonal, two
positions with the same count vector $(n_d)$ are reachable from one another by
slides alone.  So, up to reachability, a position is its count vector
$n=(n_d)$ with $0\le n_d\le L_d$, and a duplicate is the operation
\[
\mathrm{split}(d):\qquad n_d\ge 1,\ n_{d-1}\le L_{d-1}-2
\ \Longrightarrow\ n_d\mathrel{-}=1,\ n_{d-1}\mathrel{+}=2 .
\]

\begin{lemma}[Downward flow]\label{lem:U}
For every $k$, the top-cumulative weight $U_k(P)=\sum_{d\ge k}n_d 2^{d}$ is
non-increasing under moves.
\end{lemma}
\begin{proof}
Slides fix $n$.  A $\mathrm{split}(d)$ with $d>k$ keeps the weight in
$\{\ge k\}$; with $d=k$ it sends $2^{k}$ to level $k-1$, decreasing $U_k$ by
$2^{k}$; with $d<k$ it does not touch $U_k$.
\end{proof}

\begin{lemma}[Monovariant; short solutions]\label{lem:phi}
$\Phi(P)=\sum_d n_d\,d\,2^{d}$ strictly decreases by $2^{d}$ under
$\mathrm{split}(d)$.  Hence no position repeats, and since each duplicate raises
the token count by $1$ while the board holds at most $N^2$ tokens, any play uses
at most $N^2$ duplicates.
\end{lemma}
\begin{proof}
$\Delta\Phi=-d\,2^{d}+2(d-1)2^{\,d-1}=-d\,2^{d}+(d-1)2^{d}=-2^{d}<0$.
\end{proof}

\begin{corollary}\label{cor:np}
If a final position is reachable, one is reachable within $O(N^2)$ moves
($\le N^2$ duplicates and $\le N^2$ slides).  In particular the decision version
of Problem~\ref{prob} is in $\mathsf{NP}$.
\end{corollary}

\section{Final positions are rectangles, and they factor}\label{sec:fact}

For a rectangle $R\times S$ define the \emph{row value} and \emph{column
selector}
\[
V \;=\; \sum_{r\in R}2^{r}, \qquad M \;=\; \sum_{c\in S}2^{\,N-1-c}.
\]
Both lie in $[1,2^{N}-1]$; the binary digits of $V$ are $R$ and those of $M$ are
$\{N-1-c:c\in S\}$.  Note $p=|R|=\popc(V)$ and $q=|S|=\popc(M)$, where $\popc$ is
the number of $1$-bits.

The reflection $c\mapsto N-1-c$ in $M$ is exactly what aligns the polynomial product with the
board geometry. Writing $\rho(x)=\sum_{r\in R}x^{r}$ and
$\tilde\sigma(x)=\sum_{c\in S}x^{\,N-1-c}$, their product convolves to
\[
\rho(x)\,\tilde\sigma(x)\;=\;\sum_{r\in R,\ c\in S}x^{\,r-c+N-1},
\]
whose exponent $r-c+N-1$ is the diagonal $d=r-c$ shifted into $[0,2N-2]$; evaluating at $x=2$
gives $V\cdot M=W(R\times S)\,2^{N-1}$, the scaled diagonal weight $\sum 2^{\,r-c}$. This
convolution is the polynomial $P$ used in Step~B (Section~\ref{sec:algo}).

\begin{theorem}[Factorisation identity]\label{thm:fact}
Let $P=R\times S$.  Then $W(P)\cdot 2^{N-1}=V\cdot M$.  Consequently, if a final
position is reachable from $P_0$ then the \emph{scaled integer}
\[
W' \;:=\; W(P_0)\cdot 2^{N-1}\ \in\ \Z
\]
factors as $W'=V\cdot M$ with $1\le V,M\le 2^{N}-1$, and the two factors read off
the rows and columns of the target.
\end{theorem}
Throughout we keep $W(P)=\sum_{(r,c)\in P}2^{\,r-c}$ for the \emph{raw board weight}
(a conserved dyadic rational, Lemma~\ref{lem:W}) and reserve the primed symbol $W'=W\,2^{N-1}$
for its scaled \emph{integer} form---the number actually factored.
\begin{proof}
$W(R\times S)=\sum_{r\in R,\,c\in S}2^{\,r-c}
=\bigl(\sum_{r\in R}2^{r}\bigr)\bigl(\sum_{c\in S}2^{-c}\bigr)
=V\cdot\bigl(M\,2^{-(N-1)}\bigr)$.  Multiply by $2^{N-1}$.  All exponents in
$W(P_0)=\sum_d n_d 2^{d}$ satisfy $d\ge-(N-1)$, so $W'\in\Z$; and
$W'\le(2^{N}-1)^2<2^{2N}$.  Reachability preserves $W$ by Lemma~\ref{lem:W}.
\end{proof}

Thus $W'$ is a $2N$-bit number and any final position exhibits a factorisation of
it into two $N$-bit factors.

\section{Hardness: splitting balanced semiprimes}\label{sec:hard}

\begin{lemma}[Canonical maximal start]\label{lem:high}
Among all positions of a fixed weight $W$, the ``greedy-high'' position
$n^{0}$---obtained by filling diagonals from the largest $d$ downward until the
weight $W$ is exhausted---maximises $U_k$ for every $k$.  Consequently every
rectangle $R\times S$ with $W(R\times S)=W$ has $U_k(R\times S)\le U_k(n^{0})$ for
all $k$, and is reachable from $n^{0}$ by duplicates (the remaining capacity
conditions of Proposition~\ref{prop:A} hold because $R\times S$ is a genuine
rectangle).
\end{lemma}

\begin{theorem}[Hardness]\label{thm:hard}
Any algorithm that, given $P_0$, outputs a reachable final position also outputs a
factorisation $W'=V\cdot M$ with $V,M<2^{N}$ (Theorem~\ref{thm:fact}).  Moreover,
given a balanced semiprime $W'=f\cdot g$ with $2^{N-1}<f,g<2^{N}$, the greedy-high
position of weight $W'\,2^{-(N-1)}$ is solvable, and its only factorisation into
two factors below $2^{N}$ is $f\cdot g$; hence a solution reveals $\{f,g\}$.  Thus
producing a final position is at least as hard as splitting balanced semiprimes,
and an $N^{O(\log N)}$ algorithm would split $b$-bit balanced semiprimes in time
$2^{O(\log^2 b)}$.  The promise version (Problem~\ref{prob}) inherits this by
enumerating the $O(N^2)$ pairs $(p,q)=(\popc f,\popc g)$.
\end{theorem}

\section{Algorithms once a target is fixed}\label{sec:algo}

By the invariants of Section~\ref{sec:inv} we may reason purely with count vectors
(diagonal profiles) and defer the board geometry to a final round of slides, carried out in
Step~B (Section~\ref{sec:stepB}).  Suppose the target rectangle $R\times S$ (equivalently the
factorisation $W'=V\cdot M$) is known, and let $n^{*}$ be its diagonal profile,
$n^{*}_d=|\{(r,c)\in R\times S: r-c=d\}|$.

\subsection{Step A: a forced downward chip-flow (polynomial)}

Weight can only move one diagonal down, so the flow that turns $n^0$ into $n^{*}$
is forced.  Let $s_d$ be the number of tokens pushed from $d$ to $d-1$.

\begin{proposition}[Step A]\label{prop:A}
Sweeping from the top, set $s_d=n^0_d+2\,s_{d+1}-n^{*}_d$ (with $s\equiv 0$ above
the board).  The duplicates transforming $n^0$ into $n^{*}$ are feasible iff
\begin{enumerate}[label=\textup{(\roman*)},leftmargin=2.4em]
\item $s_d\ge 0$ for all $d$ \ \ (equivalently $U_k(n^{*})\le U_k(n^0)$ for all $k$);
\item $n^{*}_d\le L_d$ for all $d$;
\item $L_{d-1}\ge 2$ whenever $s_d>0$.
\end{enumerate}
When feasible, a pipelined schedule realises it with
$\sum_d s_d=pq-|n^0|\le N^2$ duplicates, in $O(N^2)$ time.
\end{proposition}
\begin{proof}[Proof sketch]
The identity $n^{*}_d=n^0_d+2s_{d+1}-s_d$ (weight in $=$ initial $+$ arrivals,
weight out $=$ pushes) determines $s$ uniquely top-down; $s_{-(N-1)}=0$ follows
from $W$-conservation.  Non-negativity of $s$ is exactly reachability by
Lemma~\ref{lem:U}; (iii) is the requirement that one may ever duplicate into a
transit diagonal; (ii) is the target capacity.  Executing highest-surplus-first,
freeing a transit diagonal downward before refilling it, never deadlocks because
every transit diagonal has $\ge 2$ cells and the target leaves room below.  The
count bound is immediate.
\end{proof}

\begin{remark}[Extreme diagonals]
The length-$1$ diagonals never receive duplicated flow. A duplicate into $d-1$ requires
$L_{d-1}\ge2$ (Definition~\ref{def:moves}), and $L_{-(N-1)}=1$, so condition~(iii) forces
$s_{-(N-2)}=0$: nothing is ever duplicated onto the bottom diagonal. Correspondingly a target
rectangle needs at most $n^{*}_{-(N-1)}=\mathbf 1[\,0\in R\text{ and }N-1\in S\,]\le
L_{-(N-1)}=1$ token there---the single cell $(0,N-1)$---which the greedy-high start already
supplies (its bottom count equals $n^{*}_{-(N-1)}$ once $s_{-(N-2)}=0$). Hence no rectangle
target violates~(iii) at the bottom. The top diagonal $d=N-1$, also of length~$1$, only
\emph{emits} flow downward and is therefore unconstrained.
\end{remark}

\begin{verbatim}
StepA(start, target):   # start = greedy-high profile, target = rectangle profile
    for d from N-1 down to -(N-1):
        s[d] = start[d] + 2*s[d+1] - target[d]        # s above the top is 0
        if s[d] < 0: return INFEASIBLE            # weight cannot flow upward
        if target[d] > L(d): return INFEASIBLE
        if s[d] > 0 and L(d-1) < 2: return INFEASIBLE
    cur = copy(start); moves = []
    while cur != target:                            # pipelined realization
        d = highest diagonal with cur[d] > target[d]
        while free(cur, d-1) < 2: push one token from d-1 downward (recurse)
        duplicate: d -> two on d-1; record move; update cur
    return moves                       # |moves| = pq - |start| <= N^2
\end{verbatim}

\subsection{Step B: seating by slides (polynomial when \texorpdfstring{$R,S$}{R,S} known)}\label{sec:stepB}

With $n^{*}$ frozen, sliding into the rectangle is a placement.  Encode the
profile as the polynomial
\[
P(x)\;=\;\sum_d n^{*}_d\,x^{\,d+N-1}\;=\;\rho(x)\,\tilde\sigma(x),\qquad
\rho(x)=\sum_{r\in R}x^{r},\quad \tilde\sigma(x)=\sum_{c\in S}x^{\,N-1-c},
\]
two $0/1$ polynomials of degree $<N$; note $P(2)=\rho(2)\tilde\sigma(2)=V\cdot
M=W'$.

\begin{proposition}[Step B]\label{prop:B}
\begin{enumerate}[label=\textup{(\roman*)},leftmargin=2.4em]
\item If $R,S$ are known, seating is $O(N^2)$: for each diagonal $d$, slide its $n^{*}_d$
tokens onto the target cells $\{(r,r-d):r\in R,\ r-d\in S\}$.
\item If only $n^{*}$ is given, recovering $(R,S)$ is the two-set turnpike (partial-digest)
problem: factor $P$ over $\Z$ in polynomial time and regroup its irreducible factors into two
$0/1$ products $\rho,\tilde\sigma$ of sizes $p,q$.
\item A dynamic program over $0/1$-valid partial products runs in $N^{O(k)}$, where $k$ is
the number of irreducible factors of $P$; this is $N^{O(\log N)}$ whenever $k=O(\log^2 N)$ and
polynomial when $k=O(1)$.
\item When $W'$ is a balanced semiprime---the hardness instances---its factors $V,M$ are
prime, so $\rho,\tilde\sigma$ are the base-$2$ expansion polynomials of $V,M$; taken at their
true degrees $\max R=\lfloor\log_2 V\rfloor$ and $\lfloor\log_2 M\rfloor$ (leading coefficient
$1$, so Cohn's hypothesis is met), they are irreducible by Cohn's theorem~\cite{bfo}. Then
$k=2$, and factoring $P$ over $\Z$ returns $(\rho,\tilde\sigma)$ outright, making Step~B
\emph{unconditionally polynomial}; large $k$ can occur only when $V$ or $M$ is composite.
\end{enumerate}
\end{proposition}

\begin{verbatim}
StepB(target):   # target = rectangle profile; returns (R,S), then seats
    # is_binary: coeffs in {0,1};  ones: #(1-coeffs);  supp: exponents present
    P = sum_d target[d] * x^(d+N-1)
    F = factor_over_Z(P)                 # polynomial time (LLL / van Hoeij)
    states = {1}                         # candidate rho, kept 0/1 throughout
    for (phi, e) in F:
        states = { g*phi^j : g in states, 0<=j<=e,
                   is_binary(g*phi^j) and ones(.) <= p and deg(.) < N }
    for rho in states with ones(rho)==p:
        st = P / rho
        if is_binary(st) and ones(st)==q:
            R = supp(rho);  S = { N-1-j : coeff_j(st)==1 }
            return seat_by_slides(R, S)  # O(N^2) slides
    return NO_RECTANGLE_WITH_THIS_PROFILE
\end{verbatim}

\section{Main theorem}\label{sec:main}

\begin{theorem}\label{thm:main}
Under the $p\cdot q$ promise, if a final position is reachable then one is
reachable within $O(N^2)$ moves, and a move sequence can be produced in time
$T_{\mathrm{fact}}(N)+O(N^2)$, where $T_{\mathrm{fact}}$ is the time to split
$W'=W(P_0)2^{N-1}$ into its two factors below $2^{N}$---exactly splitting the balanced
semiprime, and the only hard step; no sub-exponential bound for it is known. Given the
split, both phases are polynomial: Step~A by the forced downward flow, and Step~B because the
two $0/1$ factors $\rho,\tilde\sigma$ are the base-$2$ expansions of the primes $V,M$, hence
irreducible by Cohn's theorem~\cite{bfo}, so $P$ has exactly two irreducible factors that
$\Z$-factorisation recovers directly.  (For a general, non-prime balanced split the
regrouping of Proposition~\ref{prop:B} instead costs $N^{O(k)}$.)
\end{theorem}
\begin{proof}
Corollary~\ref{cor:np} bounds the length.  Compute $n^0$ from $P_0$ and
$W'=W(P_0)2^{N-1}$.  Find $V,M<2^{N}$ with $V\cdot M=W'$, $\popc V=p$, $\popc
M=q$ (this is $T_{\mathrm{fact}}$, the balanced-factoring step).  Set $R=\mathrm{bits}(V)$,
$S=\{N-1-j:\text{bit }j\text{ of }M\}$ and let $n^{*}$ be their profile.  Run
Step~A (Proposition~\ref{prop:A}) and then seat by slides
(Proposition~\ref{prop:B}); both are $O(N^2)$.  Hardness is
Theorem~\ref{thm:hard}.
\end{proof}

\begin{remark}
The exact complexity of the target-selection core, the finite-versus-asymptotic picture,
the worst-case boards of a general (non-prime) split, and the quantum status of the game are
taken up in Section~\ref{sec:open}.
\end{remark}

\section{A worked example: factoring \texorpdfstring{$143$}{143}}\label{sec:example}

Take $N=4$ and $M=143=11\cdot 13$, with $11=1011_2$ and $13=1101_2$. Laying $M$ on the
board as the conserved weight $W'=143$ by the greedy-high placement
(Lemma~\ref{lem:high}) gives the seven-token start of Figure~\ref{fig:tg143} (left).
Two duplications (Step~A) and three slides (Step~B) reach the rectangle $R\times S$ with
$R=S=\{0,1,3\}$ (right): its occupied rows give $V=2^0+2^1+2^3=11$ and its occupied
columns give $M_{\mathrm{sel}}=2^{3}+2^{2}+2^{0}=13$, so $143=11\cdot 13$. The weight
$W'=143$ is unchanged on every board.

\begin{figure}[t]
\centering
\begin{tabular}{c@{\qquad}c@{\qquad}c}
{\footnotesize initial ($W'{=}143$)} & {\footnotesize after Step~A} & {\footnotesize final rectangle}\\[3pt]
\begin{tikzpicture}[scale=0.62,baseline=(current bounding box.center)]
  \draw[tggrid] (0,0) grid (4,-4);\tglabels
  \foreach \r/\c/\w in {0/0/8,0/1/4,0/2/2,0/3/1,2/0/32,3/0/64,3/1/32}{%
    \fill[tgtok] (\c+0.5,-\r-0.5) circle (0.36);
    \node[font=\tiny,text=white] at (\c+0.5,-\r-0.5){\w};}
\end{tikzpicture}
&
\tgboard{0/0,0/1,0/2,0/3,1/0,1/1,2/0,2/2,3/0}
&
\begin{tikzpicture}[scale=0.62,baseline=(current bounding box.center)]
  \foreach \r in {0,1,3}{\fill[tgrows!14] (0,-\r) rectangle (4,-\r-1);}
  \foreach \c in {0,1,3}{\fill[tgcols!22] (\c,0) rectangle (\c+1,-4);}
  \draw[tggrid] (0,0) grid (4,-4);
  \foreach \r/\c in {0/0,0/1,0/3,1/0,1/1,1/3,3/0,3/1,3/3}{\fill[tgtok] (\c+0.5,-\r-0.5) circle (0.30);}
  \foreach \r/\w in {0/1,1/2,3/8}{\node[font=\scriptsize,text=tgrows] at (-0.5,-\r-0.5){\w};}
  \foreach \c/\w in {0/8,1/4,3/1}{\node[font=\scriptsize,text=tgcols] at (\c+0.5,-4.5){\w};}
\end{tikzpicture}
\end{tabular}
\caption{Factoring $143$ on a $4\times4$ board. \emph{Left:} $143$ laid on as weight
$W'=64{+}64{+}8{+}4{+}2{+}1=143$, each token marked with its weight $2^{\,r-c+3}$.
\emph{Step~A} (two duplications): $d{=}2$, $(3,1)\!\to\!(1,0),(2,1)$; then $d{=}1$,
$(2,1)\!\to\!(1,1),(2,2)$. \emph{Step~B} (three slides): $(0,2)\!\to\!(1,3)$,
$(2,2)\!\to\!(3,3)$, $(2,0)\!\to\!(3,1)$. \emph{Right:} the final rectangle on rows and
columns $\{0,1,3\}$; the {\color{tgrows}blue} row-weights $2^{r}$ sum to
$\color{tgrows}11$ and the {\color{tgcols}amber} column-weights $2^{\,3-c}$ (reflected) sum
to $\color{tgcols}13$, giving $143=\mathbf{\color{tgrows}11}\cdot\mathbf{\color{tgcols}13}$.}
\label{fig:tg143}
\end{figure}
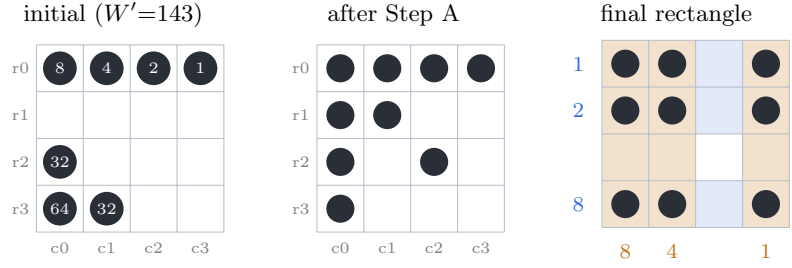

The worked example was easy only because we already knew the factorisation. In general the
one hard step is \emph{target selection}---which rectangle $R\times S$ to aim for---and the
popcount promise turns this into a \emph{bounded} combinatorial search: the rows are a
$p$-subset and the columns a $q$-subset of $[N]$, and Step~B gives a cheap, exact test of
whether a candidate profile is a genuine factorisation. A short bounded action space (the
duplications), a constrained target, and a verifiable $0/1$ terminal reward are exactly the
ingredients an AlphaZero-style learner is built for: a policy network to propose moves and a
value network to prune, wrapped in Monte-Carlo tree search over the promised targets.
Sections~\ref{sec:learn}--\ref{sec:mcts} put this to the test---first cloning Step~A with a
network, then driving the search with it---and measure how far the promise's bounded
structure lets learning and look-ahead reach. (A dual \emph{full-or-$R$} variant of the game,
which factors the board complement $A^2-W'$, is developed in \ref{sec:hash}.)

\section{A neural probe of the hardness: cloning Step~A}\label{sec:learn}

The reduction of the game to factoring is concrete enough to test empirically: one can train
a neural network to imitate Step~A move by move and ask whether it ever assembles a full
solution.\footnote{A self-contained PyTorch implementation (data generation, model, and the
metrics below) and the Java reference solver are available at
\url{https://github.com/crasmarum/TokenGame}.} It does not, and \emph{why} it does not mirrors
Theorem~\ref{thm:hard} exactly.

\paragraph{Formulation.}
A state is the diagonal profile $n\in\Z_{\ge0}^{2N-1}$ (the coefficient vector of $P$). An
action is a duplication ``at index~$i$'', $1\le i\le 2N-2$, sending $n_i\to n_i-1$ and
$n_{i-1}\to n_{i-1}+2$; a distinguished \textsc{stop} action (the ``$-1$'') marks $n=n^{*}$.
Because the weight $\sum_i n_i 2^i=W'$ is conserved and determines $n$'s instance, the map
$n\mapsto(\text{next move})$ is a well-defined function---but computing it is exactly
recovering the factorisation, so the learning target \emph{is} factoring in supervised
disguise.

\paragraph{Data.}
For random $N$-bit primes $p\le q$ we lay $C=pq$ greedy-high (Lemma~\ref{lem:high}) to get
$n^{0}$, compute the rectangle target $n^{*}$, and replay the forced-flow duplication
sequence of Proposition~\ref{prop:A}, emitting one $(\text{profile},\text{move})$ pair per
step and a terminal \textsc{stop}. One $N=64$ instance yields $\approx 10^{3}$ pairs; the
canonical order $p\le q$ fixes the target, hence the labels.

\paragraph{Model and conditioning.}
The policy is a one-dimensional residual CNN (six convolutional blocks; a per-index
$1\!\times\!1$ ``move'' head alongside a globally pooled \textsc{stop} head;
$\approx0.19$M parameters), mapping features of shape $(\text{channels},2N-1)$ to $2N$
logits. Three input regimes probe how much of the answer must be supplied:
\emph{none} (the profile only), \emph{pop} (the profile plus the promised popcounts of
$p,q$ as constant channels), and \emph{target} (the profile plus $n^{*}$ itself). Deeper
variants add capacity and, with dilated kernels, receptive field; Figure~\ref{fig:n16arch}
diagrams the largest instance---the depth-$16$ dilated $1.2$M-parameter net used at
$N=16$---including the value head that the search of Section~\ref{sec:mcts} adds.
We adopt a one-dimensional dilated residual CNN rather than an attention-based Transformer
because the board profile is an ordered $1$D grid on which chip flow is strictly local
($d\to d-1$); exponentially dilated convolutions attain a complete global receptive field
across all $2N-1$ diagonals in $\lceil\log_2(2N-1)\rceil$ layers with minimal parameters.
Because the barrier is cryptographic (Theorem~\ref{thm:hard}), replacing local inductive
biases with global self-attention cannot alter the asymptotic hardness.

\begin{figure}[t]
\centering
\includegraphics[width=\linewidth]{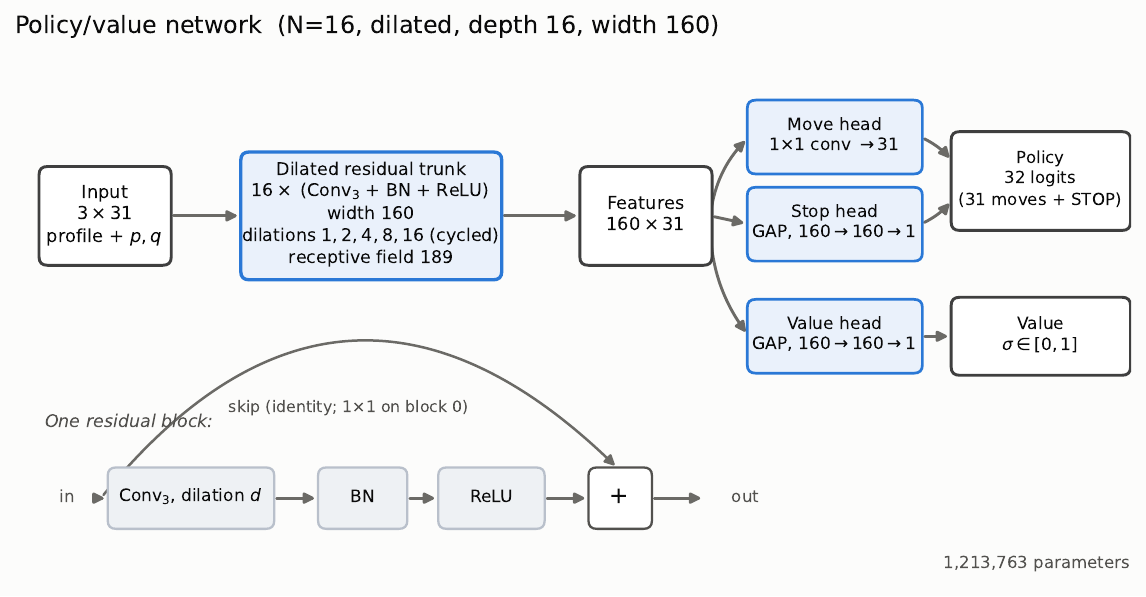}
\caption{The policy/value network, instantiated for $N=16$ (the depth-$16$ dilated,
$1.2$M-parameter checkpoint). The input channels are the profile plus the two constant
popcount channels (\emph{pop}); a stack of $16$ residual blocks
(Conv$_3$+BN+ReLU, width $160$, dilation cycling $1,2,4,8,16$) reaches receptive field
$189\gg 2N{-}1=31$, spanning the whole profile. A per-index $1{\times}1$ ``move'' head and a
globally pooled \textsc{stop} head form the $2N$ policy logits; a second pooled head gives the
scalar value ($\Pr[\,n^{*}\text{ reachable}\,]$) used by the tree search of
Section~\ref{sec:mcts}.}
\label{fig:n16arch}
\end{figure}

\paragraph{Training.}
We optimise with AdamW (weight decay $10^{-4}$) at a constant learning rate---$2\times10^{-4}$
for the largest $N=16$ run, $2\times10^{-3}$ for the smaller nets---and batch size up to
$1024$. Data is \emph{online}: each round draws $256$ fresh instances
($\approx1.8\times10^{4}$ (profile, move) pairs), regenerated every round, so there is no
fixed training set and overfitting is moot; runs last up to $4000$ rounds. The objective is
the cross-entropy of the policy over the legal duplications---with the \textsc{stop} class
up-weighted $\times20$ to offset its rarity (one per trajectory)---plus a binary
cross-entropy on the value head, in equal weight.

\paragraph{Metrics.}
\emph{Move accuracy} is the teacher-forced per-step agreement with the cloned label.
\emph{Greedy-solve} is the fraction of fresh instances for which the \emph{autonomous}
greedy rollout---$\arg\max$ over the legal actions, applied to the network's own profile
until \textsc{stop}---halts exactly at $n^{*}$.

\begin{table}[!ht]
\centering
\begin{tabular}{llccc}
\hline
$N$ & conditioning & channels & move-acc & greedy-solve\\
\hline
$8$  & pop         & $3$ & $0.97$ & $0.70$\\
$12$ & none         & $1$ & $0.79$ & $0.00$\\
$12$ & pop          & $3$ & $0.83$ & $0.00$\\
$12$ & pop, dilated & $3$ & $0.93$ & $0.08$\\
$12$ & target       & $2$ & $1.00$ & $\mathbf{1.00}$\\
$16$ & pop          & $3$ & $0.88$ & $0.00$\\
$16$ & pop, deep    & $3$ & $0.91$ & $0.00$\\
$16$ & pop, dilated & $3$ & $0.95$ & $0.05$\\
$64$ & pop          & $3$ & $0.93$ & $0.00$\\
\hline
\end{tabular}
\caption{Cloning Step~A (short runs; the $N=8,16,64$ rows use the full net and the $N=12$
rows a smaller quick net for the conditioning ablation; \emph{deep} is a $1.2$M-parameter
depth-$12$ net, $7\times$ the standard $0.19$M depth-$6$ net, and \emph{dilated} a
full-receptive-field net, depths $8$--$16$). High per-step accuracy coexists with near-zero
autonomous solves unless the target $n^{*}$ is supplied; splitting $C$ into $\{p,q\}$, which
\emph{target} hands over and \emph{pop} does not, is the entire difficulty. Full coverage
(\emph{dilated}) lifts $N=12$ and $N=16$ off $0$ (to $0.08$ and $0.05$); $0.95^{60}\approx
0.05$ shows greedy-solve still tracks $(\text{move-acc})^{\text{depth}}$, so the frontier
shifts a few bits without breaking.}
\label{tab:learn}
\end{table}

\paragraph{Reading.}
Table~\ref{tab:learn} separates the two halves of the game cleanly. \emph{target} makes the
task well-posed---each move is then the local forced-flow decision of
Proposition~\ref{prop:A}---and the network solves every instance. \emph{pop} reveals only
the rectangle's \emph{dimensions} $p\times q$, not which rows and columns; it nudges per-step
accuracy but leaves greedy-solve at $0$, because choosing the rows and columns is choosing
the factorisation. Across scale the two metrics decouple: move accuracy stays high
everywhere ($0.9$--$0.97$), while greedy-solve is nonzero only in the small-$N$
regime---$0.70$ at $N=8$---and collapses to $0$ by $N=16$--$64$. The long high-index prefix
of each trajectory is the near-trivial ``split the top diagonal'' rule, so imitating it well
says nothing about the few decisive low-index moves that encode the split; one wrong move
leaves the cloned distribution onto an unseen profile, and reaching $n^{*}$ exactly would in
any case factor a $2N$-bit balanced semiprime. The crossover near $N\approx 8$--$16$ marks
where a network of this size stops covering the factorisation. Coverage, not just capacity,
sets this frontier: giving the network a full receptive field (a dilated kernel-$3$ stack)
raises $N=12$ greedy-solve from $0.00$ to $0.08$ and lets tree search factor $23$-bit
semiprimes ($4/5$, Table~\ref{tab:mcts}), shifting the crossover outward by a few bits.
Scaling the network confirms the limit quantitatively. Deeper, better-covered nets raise
move accuracy---$0.88$ (depth $6$) to $0.91$ (depth $12$) to $0.95$ (a depth-$16$ dilated net
whose receptive field finally spans the profile)\footnote{The receptive field must cover the
whole length-$(2N-1)$ profile, because the correct move depends on cumulative sums across all
diagonals (the forced flow of Proposition~\ref{prop:A}); a plain kernel-$k$ stack reaches
this only at depth $\Theta(N/k)$, whereas kernel-$3$ blocks with exponentially growing
dilation $1,2,4,\dots$ span it in $\lceil\log_2(2N-1)\rceil$ layers---the same architecture
then serves every $N$.}---yet greedy-solve keeps tracking $(\text{move-acc})^{\text{depth}}$:
over trajectories of mean depth $\approx 60$, $0.95^{60}\approx 0.05$, matching the observed
$0.047$. One wrong move leaves the cloned distribution onto an unseen profile, so greedy play
alone stays near $0$. What the extra accuracy buys is realised only under \emph{search}: the
same depth-$16$ dilated net, with a value head and $20{,}000$ Monte-Carlo simulations,
factors $2/5$ of random $N=16$ instances (Section~\ref{sec:mcts})---including $31$-bit
semiprimes with $c[p]\cdot c[q]\sim 2\times10^{5}$ and depth-$60$ trajectories. This fits
hardness being \emph{asymptotic}: $N=16$ is still the finite, tabulatable regime
($\approx4.6$M semiprimes, Section~\ref{sec:open}), so capacity and search climb into it a
few bits at a time, while no fixed architecture escapes Theorem~\ref{thm:hard} as $N$ grows.
By Theorem~\ref{thm:main}, a classical learner that closed the move-accuracy/greedy-solve gap
at scale would be a sub-exponential factoring algorithm. Figure~\ref{fig:n16learn} shows the
full training run behind these numbers.

\begin{figure}[t]
\centering
\includegraphics[width=\linewidth]{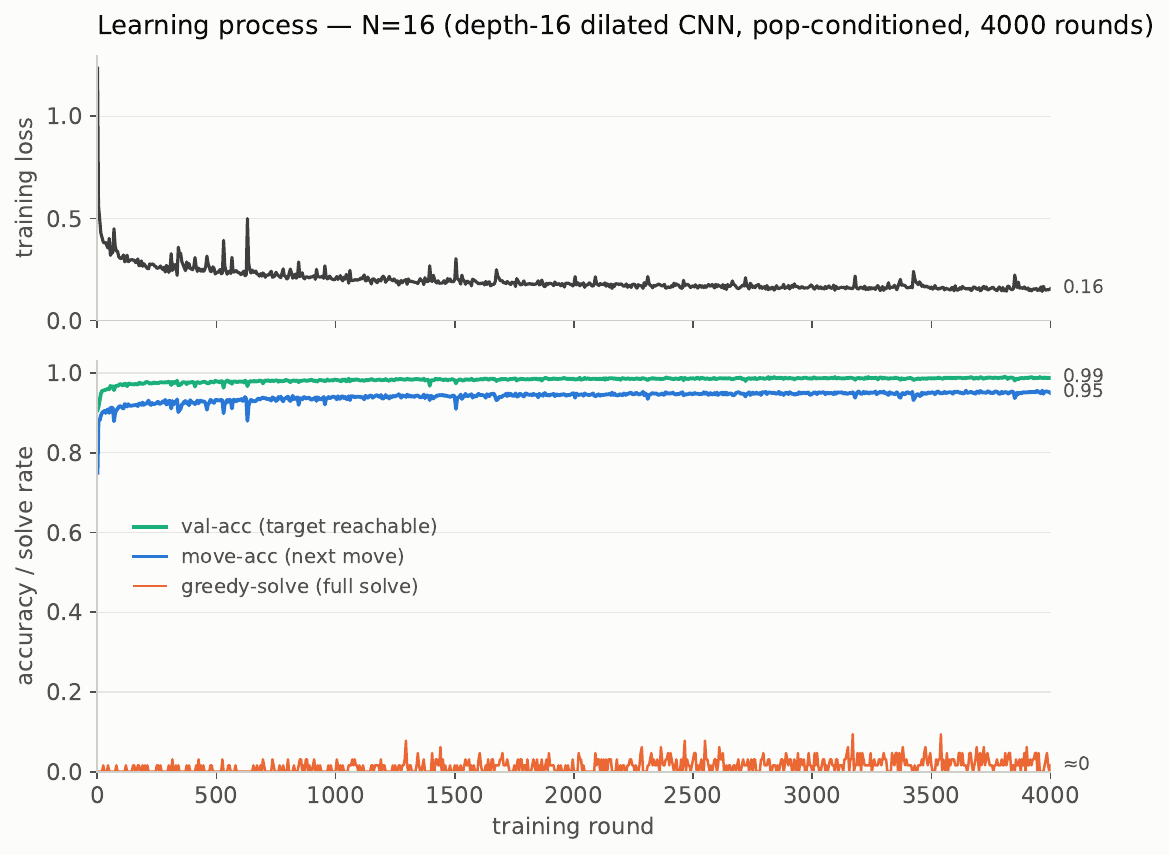}
\caption{Learning curves for the largest $N=16$ run (depth-$16$ dilated network, $1.2$M
parameters, \emph{pop} conditioning, $4000$ rounds). The training loss falls and both
per-step metrics saturate---\emph{move accuracy} at $0.95$ and \emph{val accuracy} (target
reachability) at $0.99$---while \emph{greedy-solve} (full autonomous solves) never leaves the
neighbourhood of $0$. Local imitation is learned; end-to-end solving is not, matching
$0.95^{60}\approx0.05$ and the hardness of Theorem~\ref{thm:hard}.}
\label{fig:n16learn}
\end{figure}

\section{Search with a learned prior: Monte-Carlo tree search}\label{sec:mcts}

The network of Section~\ref{sec:learn} is a pure policy; the natural strengthening is to
wrap it in look-ahead, as in AlphaZero, and let search compensate for imperfect priors. We
do, and meet the same wall from the other side.\footnote{The accompanying implementation
includes this tree search and the Step-B reward.}

\paragraph{Tree.}
A node is a profile; the root is the greedy-high placement of the number $m$ to be factored,
$n^{0}$. An action is a duplication at index~$i$. Since every move adds exactly one token,
every root-to-leaf path has the \emph{same} length $pq-\sum_i n^{0}_i$ and every leaf carries
$pq$ tokens---the promised final count.

\paragraph{Reward (Step~B).}
At a leaf we run Step~B (Proposition~\ref{prop:B}): factor the leaf polynomial into two
$0/1$ polynomials $\rho,\tilde\sigma$ of sizes $p,q$, and return $1$ iff this succeeds with
$\rho(2)\tilde\sigma(2)=m$, else $0$. Step~B is polynomial, so the reward is cheap to
evaluate; the difficulty is \emph{reaching} a rewarding leaf, since among the many
$pq$-token profiles of weight $m$ only the (essentially unique) rectangle profile of $m$'s
factorisation scores $1$.

\paragraph{Selection (PUCT).}
Descend by
\[
a^{*}=\arg\max_{a}\ Q(s,a)+c_{\mathrm{puct}}\,P(s,a)\,
\frac{\sqrt{\sum_b N(s,b)}}{1+N(s,a)},
\]
where the prior $P(s,\cdot)$ is the policy network of Section~\ref{sec:learn} softmaxed over
the legal duplications---the ``additional weight'' the model places on moves. We take
$c_{\mathrm{puct}}=1.5$ and a budget of $2000$--$20{,}000$ simulations per move (the $N=16$
results use $20{,}000$); each simulation expands exactly one leaf, so there is no separate
node-count cap. The search then commits the most-visited move and recurses.

\paragraph{Value head.}
A non-terminal leaf needs a value. Two options coincide in the limit. A policy-guided
\emph{rollout} to a terminal backs up the exact $0/1$ Step-B reward---no approximation error,
but an $O(N^2)$ playout per simulation. Alternatively we add a \emph{value head} to the
network: a second output $\sigma(v)\approx\Pr[\,n^{*}\text{ reachable from this profile}\,]$.
Reachability is precisely the $0/1$ value---a profile that still majorises $n^{*}$ (a
nonnegative forced flow, the Step-A feasibility test of Proposition~\ref{prop:A}) yields
reward $1$ under optimal play, otherwise $0$---so the head is trained by the very reduction it
accelerates, on the cloned trajectories (reachable, label $1$) together with random off-path
profiles that have left the feasible cone (label $0$). It reaches about $0.90$ reachability
accuracy at $N=8$ and replaces each rollout with a single forward pass, so a fixed compute
budget buys many more simulations. Either way the reward is exact at the leaves that matter,
and the only obstacle is the size of the tree.

\paragraph{Findings.}
At small $N$ the search solves outright, recovering the primes, and the learned prior
sharpens a uniform one (Table~\ref{tab:mcts}): at $N=8$ the bare greedy policy already solves
about $0.70$ of instances (Table~\ref{tab:learn}), and wrapping it in tree search lifts that
to $5/5$ (the value head matches this while evaluating leaves in one forward pass rather than
a rollout). Capacity and search then push the frontier through the small sizes: with a
full-coverage dilated net they clear $N=12$ ($4/5$, $23$-bit semiprimes), and a depth-$16$
dilated net with the value head and $20{,}000$ simulations reaches $N=16$ itself---$2/5$ of
random $31$--$32$-bit instances, including hard ones with candidate counts
$c[p]\cdot c[q]\sim 2\times10^{5}$ and depth-$60$ trajectories (e.g.\ $34589\cdot44711$),
where the bare policy is near $0$. This answers the natural question---\emph{would a deeper
network clear $N=16$?}---in the affirmative, but for a telling reason: $N=16$ is still the
\emph{finite, tabulatable} regime ($\approx4.6$M semiprimes, Section~\ref{sec:open}), so a
large enough model and search climb into it. The wall is structural but asymptotic: the
leaves are the $pq$-token profiles of weight $m$, a set that grows while the rewarding subset
stays of size $\approx 1$, so locating it is the target-selection of
Theorem~\ref{thm:hard}. A larger prior or more simulations shift the frontier outward by a
bounded number of bits---each added bit roughly doubling the tree---so the gain is a
constant-factor edge that is exhausted once $N$ leaves the tabulatable range ($N\gtrsim32$);
by Theorem~\ref{thm:main}, an MCTS that solved at scale would split balanced semiprimes in
sub-exponential time.

\begin{table}[t]
\centering
\begin{tabular}{lllc}
\hline
$N$ & prior & sims/move & solve rate\\
\hline
$6$  & uniform                              & $6000$  & $3/5$\\
$6$  & learned (\emph{pop})                 & $6000$  & $5/5$\\
$8$  & learned (\emph{pop})                 & $2000$  & $5/5$\\
$12$ & learned (\emph{pop, dilated})        & $20000$ & $4/5$\\
$16$ & learned (\emph{pop})                 & $2000$  & $0/5$\\
$16$ & learned (\emph{pop, dilated depth 16}) & $20000$ & $2/5$\\
\hline
\end{tabular}
\caption{Monte-Carlo tree search on random semiprimes (five instances each; illustrative).
The learned prior lifts the solve rate at $N=6$, clears $N=8$ ($16$-bit semiprimes, e.g.\
$167\cdot211$) and, with a full-coverage dilated net and a value head, $N=12$ ($23$-bit,
e.g.\ $2213\cdot2293$). A small net leaves $N=16$ at $0/5$, but a depth-$16$ dilated net with
the value head and $20{,}000$ simulations reaches $2/5$ of $31$--$32$-bit instances (e.g.\
$34589\cdot44711$, with $c[p]\cdot c[q]\approx2\times10^{5}$)---$N=16$ being the last
tabulatable size; the wall is asymptotic ($N\gtrsim32$). The tabulated rates are
five-instance point estimates; a larger $50$-instance sweep at $N=8$ gives $49/50=0.98$
(Wilson $95\%$ CI $[0.90,1.00]$), calibrating the $5/5$ entry. Matched $50$--$100$-instance
sweeps at $N=12,16$ (at $20{,}000$ simulations per move) are compute-intensive and left to
GPU runs.}
\label{tab:mcts}
\end{table}

\section{Discussion and open problems}\label{sec:open}

Steps~A and~B are polynomial, so the solver confines all difficulty to one
number-theoretic split (Theorem~\ref{thm:main}); four questions sharpen its status.

\paragraph{Finite versus asymptotic hardness.}
That the experiments of \S\ref{sec:learn}--\S\ref{sec:mcts} factor instances up to $N=16$ does
not contradict the hardness above, because the difficulty is asymptotic. For fixed $N$ the
instance space is finite and, for small $N$, tiny: there are only $\pi(2^{N})-\pi(2^{N-1})$
primes of exactly $N$ bits---e.g.\ $3030$ at $N=16$, hence $\binom{3030}{2}\approx4.6$M
balanced semiprimes, a table one could precompute. The $\popc$ promise shrinks each instance
further to the $c[p]\,c[q]$ prime pairs of the promised popcounts, where by the prime number
theorem
\[
c[p]\ \approx\ \frac{\binom{N-1}{p-1}}{N\ln 2}
\quad\Bigl(\text{or }\tfrac{2\binom{N-2}{p-2}}{N\ln 2}\text{ forcing the odd bit}\Bigr),
\qquad
c[p]\,c[q]\ \approx\ \frac{\binom{N-1}{p-1}\binom{N-1}{q-1}}{(N\ln 2)^2}.
\]
The numerator counts $N$-bit integers of popcount $p$ exactly, and $1/(N\ln 2)$ is the prime
density near $2^{N}$; the odd-bit variant shifts the popcount peak up by one, as the data
show. For balanced popcounts $p\approx q\approx N/2$ this is $\Theta(2^{2N}/N^{3})$: the
promise removes only a polynomial factor and the search stays exponential---whereas a rare
popcount makes $\binom{N-1}{p-1}$ small (a handful of candidates), which is exactly why those
instances are easy. A large enough model with enough search therefore
climbs into the small-$N$ regime; but each added bit doubles the primes, the semiprimes, and
the tree, the tabulatable range ends near $N\approx32$ ($\sim2\times10^{8}$ primes,
$\sim10^{16}$ semiprimes), and beyond it Theorem~\ref{thm:hard} bites.

\paragraph{(1) Counting $0/1$-polynomial divisors.}
Step~B must single out the divisors of $P$ that themselves have coefficients in
$\{0,1\}$. The crude estimate is the $2^{k}$ subset-products of the $k$ irreducible
factors of $P$, but genuine $0/1$ divisors are far rarer than arbitrary subset-products.
We conjecture that \emph{every degree-$D$ integer polynomial has at most $D^{O(\log D)}$
divisors in $\{0,1\}[x]$}. A proof would let Step~B enumerate them directly in
quasi-polynomial time, making target selection---and hence the whole
solver---\emph{unconditionally} $N^{O(\log N)}$, with no hypothesis on the factor count
$k$. The bound holds trivially when $k=O(\log N)$ (then even $2^{k}$ is polynomial) and is
false for unrestricted coefficients; the gap between these is the open content, and by
item~(3) it is exactly what the worst-case boards exercise.

\paragraph{(2) Exact complexity of the target-selection core.}
Recovering $(R,S)$ from the profile $n^{*}$ alone (Proposition~\ref{prop:B}) is the
factorisation of $P=\rho\,\tilde\sigma$ into two $0/1$ polynomials---a partial-digest /
turnpike-type reconstruction of two sets. Under the substitution $x\mapsto2$ it becomes
splitting the balanced semiprime $M=\rho(2)\,\tilde\sigma(2)$, so the search problem is
integer factoring in disguise. Its decision version---``does $M$ have a factor in
$[2^{N-1},2^{N})$ of popcount $p$?''---lies in $\mathsf{NP}\cap\mathsf{coNP}$: the prime
factorisation of $M$ together with Pratt primality certificates~\cite{pratt} witnesses
\emph{both} the yes- and the no-answer. This is strong evidence that the core is \emph{not}
$\mathsf{NP}$-complete, since that would force $\mathsf{NP}=\mathsf{coNP}$; symmetrically,
a polynomial-time algorithm for it would split balanced semiprimes classically. The core
therefore sits squarely in the factoring-flavoured ``intermediate'' regime, and deciding
between $\mathsf{P}$ and $\mathsf{NP}$-completeness is tantamount to resolving the status of
factoring itself---the precise sense in which its class is open. (For contrast, the
\emph{classical} turnpike problem---reconstruct points on a line from their multiset of
pairwise differences---admits only a pseudo-polynomial backtracking
algorithm~\cite{ssl} and has itself resisted classification for decades.)

\paragraph{(3) Which boards are worst-case (the general split).}
By Proposition~\ref{prop:B} the prime promise forces $k=2$ (Cohn's theorem), so a large
factor count is a feature of the \emph{general} balanced split---where $V,M$ may be
composite---not of the semiprime instances. There the dynamic program runs in $N^{O(k)}$
with $k$ the number of irreducible factors of $P$, and large $k$ arises precisely when
$\rho$ or $\tilde\sigma$ is cyclotomic-rich, i.e.\ when a \emph{composite} factor has a
highly structured binary expansion. The extremal case is a composite Mersenne-type factor
$V=2^{m}-1$ with $m$ composite: then $R=\{0,\dots,m-1\}$ is contiguous and
\[
\rho(x)=1+x+\cdots+x^{m-1}=\prod_{d\mid m,\ d>1}\Phi_d(x)
\]
splits into $\tau(m)-1$ irreducible cyclotomic factors ($\tau$ the divisor function; a
Mersenne \emph{prime} has $m$ prime and $\rho=\Phi_m$ irreducible). Arithmetic-progression
and base-$b$ repunit bit-sets behave the same way. Since $\tau(m)$ is
$2^{\Omega(\log m/\log\log m)}$, such $k$ \emph{exceeds every fixed power of $\log N$}, so on
these boards the naive $2^{k}$ enumeration is super-quasi-polynomial---exactly the regime
that the $0/1$-divisor bound of item~(1) would defuse. On prime factors, by contrast, Cohn's
theorem keeps $\rho$ irreducible, $P$ two-factored, and Step~B polynomial.

\paragraph{(4) Quantum status.}
The game's hardness core is balanced-semiprime splitting, and integer factoring is in
$\mathsf{BQP}$ by Shor's algorithm~\cite{shor}---balancedness neither helps nor hinders
Shor, which factors any composite in polynomial time. Hence target selection is in $\mathsf{BQP}$: a
quantum machine factors $M=pq$ in time $\mathrm{poly}(N)$ and, composed with the classical
polynomial Steps~A and~B, outputs the full $O(N^2)$ move sequence in quantum polynomial
time. The decision problem ``is a final position of at most $p\cdot q$ tokens reachable?''
is therefore in $\mathsf{BQP}$, and the game is \emph{not} a candidate for post-quantum
hardness: its apparent classical intractability is exactly as quantum-fragile as the
factoring problem it encodes.

\paragraph{Constraint-satisfaction view.}
A final position is a combinatorial rectangle $R\times S$, i.e.\ a \emph{biclique} in the
$N\times N$ board seen as a bipartite graph on rows and columns; target selection asks for a
biclique of prescribed weight $V\cdot M=W'$ and prescribed dimensions $|R|=p$, $|S|=q$. This
places the game in the family of weighted-biclique and monochromatic-rectangle problems---the
maximum edge biclique problem is $\mathsf{NP}$-hard, and $0/1$-matrix rectangles are the
atoms of communication-complexity lower bounds---so the promised problem is naturally read as
a constraint-satisfaction instance: choose $p$ rows and $q$ columns whose weighted product
hits a target. What distinguishes our instance from generic biclique/SAT search is the
number-theoretic weight constraint, which (Theorem~\ref{thm:hard}) is precisely balanced
factoring; the popcount promise fixes the biclique's side lengths, and the diagonal invariant
turns ``reach a monochromatic rectangle'' into ``split $W'$''. A tighter reduction to or from
standard weighted-CSP encodings (independent set, exact cover) would connect the game to the
broader constraint-satisfaction toolbox.

\appendix
\renewcommand{\thesection}{Appendix~\Alph{section}}

\section{Binary operation on weights, and the full-or-\texorpdfstring{$R$}{R} variant}\label{sec:hash}

Consider the variant in which a position is \emph{final} when every column is either
\emph{full} (the all-ones pattern $11\cdots1$, value $A:=2^{N}-1$) or equal to one common
pattern $R$. Putting $R$ on every column and then completing the full columns gives, for
$U$ the (reflected) selector of the full columns and $V_R$ the value of $R$,
\[
W' \;=\; A\cdot U + V_R\,(A-U).
\]
This is one instance of a natural bitwise operation, which we isolate and analyse; the
complementary game it governs is illustrated in Figure~\ref{fig:tg82}.

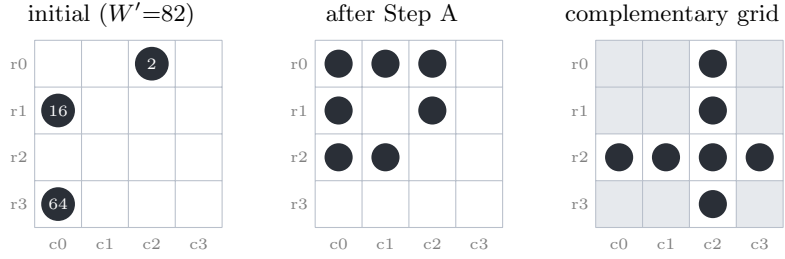
\begin{figure}[t]
\centering
\begin{tabular}{c@{\qquad}c@{\qquad}c}
{\footnotesize initial ($W'{=}82$)} & {\footnotesize after Step~A} & {\footnotesize complementary grid}\\[3pt]
\begin{tikzpicture}[scale=0.62,baseline=(current bounding box.center)]
  \draw[tggrid] (0,0) grid (4,-4);\tglabels
  \foreach \r/\c/\w in {3/0/64,1/0/16,0/2/2}{%
    \fill[tgtok] (\c+0.5,-\r-0.5) circle (0.36);
    \node[font=\tiny,text=white] at (\c+0.5,-\r-0.5){\w};}
\end{tikzpicture}
&
\tgboard{0/0,0/1,0/2,1/0,1/2,2/0,2/1}
&
\begin{tikzpicture}[scale=0.62,baseline=(current bounding box.center)]
  \foreach \r in {0,1,3}{\foreach \c in {0,1,3}{\fill[tggrid!35] (\c,-\r) rectangle (\c+1,-\r-1);}}
  \draw[tggrid] (0,0) grid (4,-4);\tglabels
  \foreach \r/\c in {0/2,1/2,2/0,2/1,2/2,2/3,3/2}{\fill[tgtok] (\c+0.5,-\r-0.5) circle (0.30);}
\end{tikzpicture}
\end{tabular}
\caption{The complementary (full-or-$R$) game on $143$, with $A^2-143=225-143=82$.
\emph{Left:} $82=64{+}16{+}2$ laid on the board. \emph{Step~A} (four duplications):
$(3,0)\!\to\!(2,0),(3,1)$; $(3,1)\!\to\!(2,1),(3,2)$; $(3,2)\!\to\!(0,0),(1,1)$;
$(1,1)\!\to\!(0,1),(1,2)$. \emph{Step~B} (three slides): $(0,1)\!\to\!(2,3)$,
$(0,0)\!\to\!(2,2)$, $(1,0)\!\to\!(3,2)$. \emph{Right:} the final full-or-$R$ position
---column~$2$ full, columns $\{0,1,3\}$ equal to $R'=\{2\}$--- is exactly the
occupied\,$\leftrightarrow$\,empty complement of Figure~\ref{fig:tg143}'s rectangle
(shaded cells).}
\label{fig:tg82}
\end{figure}

\begin{definition}\label{def:hash}
For $N$-bit integers $U,W\in\{0,\dots,2^N-1\}$ with $A:=2^N-1$, define
\[
U\,\#\,W \;:=\; A\,U + W\,(A-U) \;=\; A(U+W)-UW \;=\; A^2-(A-U)(A-W).
\]
\end{definition}

\begin{proposition}[bit-serial form]\label{prop:hash}
$U\#W$ is computed by scanning the $N$ bit-positions of $U$ and emitting, at position $i$,
the block $W$ if bit $i$ is $0$ and the all-ones block $A$ if it is $1$, then summing the
shifted blocks: $U\#W=\sum_{i=0}^{N-1}c_i2^i$, $c_i\in\{W,A\}$. Equivalently, with a
width-$0$ terminator returning $0$ and $U=2A'+a_0$,
\[
U\,\#\,W \;=\; 2\,(A'\,\#\,W) \;+\; \bigl(a_0=0\,?\,W:\,A\bigr),
\]
run over all $N$ positions.
\end{proposition}
\begin{proof}
Since $\sum_{i:\,u_i=1}2^i=U$ and $\sum_{i:\,u_i=0}2^i=A-U$, we get
$\sum_i c_i 2^i = A\,U+W(A-U)$, which is Definition~\ref{def:hash}. The terminator must be
$0$: run at full width, an all-zeros $U$ then emits $W$ at every position and gives
$\sum_{i<N}W2^i=A\,W$; a $W$-terminator would add a spurious $2^N W$, and would also drop
the leading-zero positions of $U$ (which the game needs, as $U=M_F$ may have them).
\end{proof}

\begin{corollary}\label{cor:hash}
With $p=U/A,\ q=W/A\in[0,1]$, $\;(U\#W)/A^2 = p+q-pq = 1-(1-p)(1-q)$. Hence $\#$ is
commutative ($U\#W=W\#U$); it is the De~Morgan dual of multiplication under the complement
$\overline{X}=A-X$; scaled by $A^2$ it is the probabilistic OR (a $t$-conorm); it is
monotone nondecreasing in each argument ($\partial(U\#W)/\partial U=A-W\ge0$); $A$ is
absorbing ($X\#A=A^2$); and $0\#W=A\,W$.
\end{corollary}

\begin{theorem}[the variant factors the complement]\label{thm:hashfact}
The full-or-$R$ final positions of weight $W'$ are exactly those with $W'=U\,\#\,V_R$,
equivalently
\[
(A-U)(A-V_R)\;=\;A^2-W'.
\]
Thus a final position exists iff $A^2-W'=(2^N-1)^2-W'$ factors into two $N$-bit factors
$\overline U,\overline{V_R}$. Consequently: \emph{(i)} computing $W'=U\#W$ is two
multiplications; \emph{(ii)} given $W'$ and one operand $U\neq A$, the other is a single
division $W=(W'-AU)/(A-U)$; \emph{(iii)} recovering both operands from $W'$ alone is
splitting the $2N$-bit number $A^2-W'$ --- the balanced-factoring wall of
Section~\ref{sec:hard}, now applied to the board complement rather than to $W'$.
\end{theorem}

\begin{remark}\label{rem:menu}
So the ``empty-or-$R$'' game (Section~\ref{sec:fact}) factors $W'$, while the
``full-or-$R$'' game factors $A^2-W'$: the two winning conditions are exchanged by
complementing the board (occupied $\leftrightarrow$ empty), which sends
$W'\mapsto A^2-W'$; Steps~A and~B are unchanged, only the target number differs.
Reducing $A^2-W'\equiv -W'\pmod{A}$ loses everything, since reduction mod $2^N-1$ is a
homomorphism. More generally, allowing a fixed menu of $k$ column patterns
$\{R_1,\dots,R_k\}$ makes $W'=\sum_i V_{R_i}M_i$, a sum of $k$ products; $k=1$ is
Section~\ref{sec:fact}, and $k=2$ with one pattern full is the operation~$\#$.
\end{remark}

\medskip
Concretely, run the complementary game on the instance of Section~\ref{sec:example}. The
full $4\times4$ board has weight $A^2=(2^4-1)^2=15^2=225$, so complementing $W'=143$ gives
$A^2-W'=225-143=82$ (the reference here is the full-board weight $15^2=225$, not
$2^{2N}-1$). Laying $82$ and playing the full-or-$R$ game reaches the board-complement of
the rectangle of Figure~\ref{fig:tg143}: a ``plus'' with column~$2$ full and columns
$\{0,1,3\}$ carrying the lone pattern $R'=\{2\}$ (Figure~\ref{fig:tg82}). Its selectors are
$M_F=2^{\,3-2}=2$ and $V_{R'}=2^2=4$, and indeed
$(A-M_F)(A-V_{R'})=(15-2)(15-4)=13\cdot 11=A^2-82=143$ --- the same $11\cdot 13$, read
through the complement.

\paragraph{The geometry is easy; the target is hard.}
The full-or-$R$ variant exposes a clean decoupling in the game's complexity. Complementing the
winning condition radically alters the number-theoretic target---from $W'$ to $A^2-W'$---while
leaving the geometric mechanics (Steps~A and~B) completely unchanged. The physical constraints
of the board are therefore computationally easy: whatever the target number, duplications and
slides reach it in polynomial time \emph{once it is known}. The intractable core of the token
game is purely the algebraic search for that target.

More broadly, the operation $\#$ recasts the board as a visual calculator for polynomial
arithmetic. One column pattern makes $W'$ a single product $V\cdot M$ (Section~\ref{sec:fact});
$k$ patterns make it a sum of $k$ products (Remark~\ref{rem:menu}); the full-or-$R$ condition
realises the De~Morgan dual $A^2-(A-U)(A-W)$. In every case the board's rules are
\emph{evaluating a polynomial}, and the ``gameplay''---driving tokens to a final
position---is the search for its factorisation: for the roots that split the target over the
promised $\{0,1\}$ alphabet.

\end{document}